\documentclass[11pt,leqno]{article}
\usepackage[doc]{optional}
\usepackage{enumitem}

\usepackage{wrapfig}
\usepackage{color}
\usepackage{float}
\usepackage{soul}
\usepackage{graphicx}
\usepackage{graphics}
\definecolor{labelkey}{rgb}{0,0.08,0.45}
\definecolor{refkey}{rgb}{0,0.6,0.0}

\definecolor{Brown}{rgb}{0.45,0.0,0.05}
\definecolor{lime}{rgb}{0.00,0.8,0.0}
\definecolor{lblue}{rgb}{0.5,0.5,0.99}
\definecolor{lblue}{rgb}{0.8,0.85,1.00}
\definecolor{anotherblue}{rgb}{.8, .8,1}
\definecolor{violet}{rgb}{0.9,0.6,0.9}
\definecolor{greenyellow}{rgb}{0.53,0.99,0.18}

\definecolor{Lyellow}{rgb}{0.87,0.87,0.87}
\definecolor{Lgray}{rgb}{0.92,0.92,0.92}
\definecolor{Mgray}{rgb}{0.5,0.5,0.5}
\definecolor{Gold}{rgb}{0.99,0.84,0.0}

\definecolor{myblue}{rgb}{.8, .8, 1}
  \newcommand*\mybluebox[1]{%
    \colorbox{myblue}{\hspace{1em}#1\hspace{1em}}}

\usepackage{mathpazo}
\usepackage{amsmath}
\usepackage{amssymb}
\usepackage{theorem}
\usepackage{fancyhdr}
\usepackage{empheq}

\usepackage[margin=1.0in]{geometry}

\usepackage{hyperref}

\newcommand{\nnn}{\ensuremath{{n\in{\mathbb N}}}}

\newcommand{\menge}[2]{\big\{{#1}~\big |~{#2}\big\}}

\newcommand{\fenv}[1]%
{\ensuremath{\,\overrightarrow{\operatorname{env}}_{#1}}}
\newcommand{\benv}[1]%
{\ensuremath{\,\overleftarrow{\operatorname{env}}_{#1}}}

\newcommand{\scal}[2]{\left\langle{#1},{#2}  \right\rangle}

\newcommand{\zeroun}{\ensuremath{\left]0,1\right[}}
\newcommand{\RR}{\ensuremath{\mathbb R}}

\newcommand{\NN}{\ensuremath{\mathbb N}}

\newcommand{\RP}{\ensuremath{\mathbb R}_+}

\newcommand{\Fix}{\ensuremath{\operatorname{Fix}}}

\newcommand{\Id}{\ensuremath{\operatorname{Id}}}

\newcommand{\Up}{\ensuremath{{{U^\perp}}}}
\newcommand{\Vp}{\ensuremath{{{V^\perp}}}}

\newcommand{\pinf}{\ensuremath{+\infty}}
\newcommand{\bx}{\ensuremath{\mathbf{x}}}
\newcommand{\bX}{\ensuremath{{\mathbf{X}}}}
\newcommand{\bU}{\ensuremath{{\mathbf{U}}}}
\newcommand{\bV}{\ensuremath{{\mathbf{V}}}}

\newcommand{\bT}{\ensuremath{{\mathbf{T}}}}

\def\ve{\varepsilon}

{\begin{list}{}{%
\settowidth{\labelwidth}{\textrm{#1~}}%
\setlength{\leftmargin}{\labelwidth+\labelsep}}}
{\end{list}}
\newtheorem{theorem}{Theorem}[section]
\newtheorem{lemma}[theorem]{Lemma}
\newtheorem{corollary}[theorem]{Corollary}
\newtheorem{proposition}[theorem]{Proposition}
\newtheorem{definition}[theorem]{Definition}
\theoremstyle{plain}{\theorembodyfont{\rmfamily}
}
\theoremstyle{plain}{\theorembodyfont{\rmfamily}
}
\theoremstyle{plain}{\theorembodyfont{\rmfamily}
}
\theoremstyle{plain}{\theorembodyfont{\rmfamily}
}
\newtheorem{fact}[theorem]{Fact}
\theoremstyle{plain}{\theorembodyfont{\rmfamily}
}

\allowdisplaybreaks

\begin{document}

\title{\textrm{The rate of linear convergence of the 
Douglas--Rachford algorithm for subspaces
is the cosine of the Friedrichs angle}}

\author{
Heinz H.\ Bauschke\thanks{Mathematics, University of British
Columbia, Kelowna, B.C.\ V1V~1V7, Canada. E-mail:
\texttt{heinz.bauschke@ubc.ca}.},~
J.Y.\ Bello Cruz\thanks{IME, Federal University of Goias,
Goiania, G.O. 74001-970, Brazil. E-mail:
\texttt{yunier.bello@gmail.com}.},~Tran T.A. Nghia\thanks{Mathematics, University of British Columbia, Kelowna, B.C.\ V1V~1V7, Canada. E-mail: \texttt{nghia.tran@ubc.ca}.},~
Hung M.\ Phan\thanks{Mathematics, University of British Columbia, Kelowna, B.C.\ V1V~1V7, Canada. E-mail:  \texttt{hung.phan@ubc.ca}.}, ~and
Xianfu\ Wang\thanks{Mathematics, University of British Columbia,
Kelowna, B.C.\ V1V~1V7, Canada. E-mail:
\texttt{shawn.wang@ubc.ca}.}}

\date{December 20, 2013}

\maketitle \thispagestyle{fancy}

\vskip 8mm

\begin{abstract} \noindent
The Douglas--Rachford splitting algorithm is a 
classical optimization method that has found many applications. 
When specialized to two normal cone operators, it yields
an algorithm for finding a point in the intersection of
two convex sets. 
This method for solving feasibility problems 
has attracted a lot of attention due to its 
good performance even in nonconvex settings. 

In this paper, 
we consider the Douglas--Rachford algorithm
for finding a point in the intersection of two subspaces.
We prove that the method converges strongly to the projection of the
starting point onto the intersection. Moreover, 
if the sum of the two subspaces is closed,
then the convergence is linear with the rate being the cosine 
of the Friedrichs angle between the subspaces. 
Our results improve upon existing results in three ways:
First, we identify the location of the limit and thus reveal
the method as a best approximation algorithm;
second, we quantify the rate of convergence, and third,
we carry out our analysis in general (possibly
infinite-dimensional) Hilbert space. 
We also provide various examples as well as a 
comparison with the classical method of alternating projections.
\end{abstract}

{\small \noindent {\bfseries 2010 Mathematics Subject
Classification:} {Primary 49M27, 65K05, 65K10; 
Secondary 47H05, 47H09, 49M37, 90C25. 
}
}

\noindent {\bfseries Keywords:}
Douglas--Rachford splitting method, 
firmly nonexpansive,
Friedrichs angle, 
linear convergence,
method of alternating projections, 
normal cone operator, 
subspaces,
projection operator. 


\section{Introduction}

Throughout this paper,
we assume that
\begin{empheq}[box=\mybluebox]{equation}
\text{$X$ is a real Hilbert space}
\end{empheq}
with inner product $\scal{\cdot}{\cdot}$ and induced norm
$\|\cdot\|$. 
Let $U$ and $V$ be closed convex subsets of $X$ such that
$U\cap V\neq\varnothing$.
The \emph{convex feasibility problem} is to find a point in
$U\cap V$. This is a basic problem in the natural sciences and
engineering (see, e.g., \cite{bb96}, \cite{CenZen}, and \cite{Comb93}) --- 
as such, a plethora of algorithms based on the
nearest point mappings (projectors) $P_U$ and $P_V$ have been
proposed to solve it. 

One particularly popular method is the \emph{Douglas--Rachford
splitting algorithm} \cite{DougRach} which utilizes the 
the \emph{Douglas--Rachford splitting operator}
\begin{empheq}[box=\mybluebox]{equation}
\label{e:T}
T := P_V(2P_U-\Id) + \Id- P_U
\end{empheq}
and $x_0\in X$ to generate the sequence $(x_n)_\nnn$ by 
\begin{equation}
\label{e:DR}
(\forall\nnn) \quad x_{n+1} := Tx_n.
\end{equation}
While the sequence $(x_n)_\nnn$ may or may not converge to a
point in $U\cap V$, the projected (``shadow'') sequence
\begin{equation}
(P_Ux_n)_\nnn
\end{equation}
always converges (weakly) to a point in $U\cap V$
(see \cite{LM}, \cite{a-loch}, \cite{BauschkeJonFest}, \cite{BC2011}). 
The Douglas--Rachford algorithm has been applied very
successfully to various problems where $U$ and $V$ are not
necessarily convex, even though the supporting formal theory is
far from being complete (see, e.g., \cite{ABT}, \cite{JOSA}, and 
\cite{Elser}). 
Very recently, Hesse, Luke and Neumann \cite{HLN} (see also
\cite{HL}) considered
projection methods for the (nonconvex) sparse affine feasibility
problem. Their paper highlights the importance of understanding
the Douglas--Rachford algorithm for the case when $U$ and $V$
are closed \emph{subspaces} of $X$; their basic convergence
result is the following.

\begin{fact}[Hesse--Luke--Neumann]
\label{f:HLN}
{\rm (See \cite[Theorem~4.6]{HLN}.)}
Suppose that $X$ is finite-dimensional and 
$U$ and $V$ are subspaces of $X$. 
Then the sequence $(x_n)_\nnn$ generated by \eqref{e:DR} converges to a
point in $\Fix T$ with a linear rate\footnote{Recall that
$x_n\to x$ \emph{linearly} or with a \emph{linear rate}
$\gamma\in\zeroun$ if $(\gamma^{-n}\|x_n-x\|)_\nnn$ is bounded.}. 
\end{fact}

The aim of this paper is three-fold.
We complement Fact~\ref{f:HLN} by providing the following:
\begin{itemize}
\item
We identify the limit of the shadow sequence $(P_Ux_n)$ as
$P_{U\cap V}x_0$; consequently and somewhat surprisingly, 
the Douglas--Rachford method in this
setting not only solves a feasibility problem but actually a
\emph{best approximation problem}.
\item 
We quantify the rate of convergence --- it turns out to be the cosine of the
Friedrichs angle between $U$ and $V$; moreover, our estimate is sharp. 
\item 
Our analysis is carried out in general (possibly
infinite-dimensional) Hilbert space. 
\end{itemize}

The paper is organized as follows. 
In Sections~\ref{s:aux} and \ref{s:static}, 
we collect various auxiliary results to
facilitate the proof of the main results (Theorem~\ref{t:main} and Theorem~\ref{t:main2}) 
in Section~\ref{s:main}. 
In Section~\ref{s:plane}, we analyze the Douglas--Rachford algorithm 
for two lines in the Euclidean plane.
The results obtained are used in Section~\ref{s:anex} for
an infinite-dimensional construction illustrating the lack of 
linear convergence. 
In Section~\ref{s:compare}, we compare
the Douglas--Rachford algorithm to the method of alternating
projections. 
We report on numerical experiments in Section~\ref{s:numerical},
and conclude the paper in Section~\ref{s:conclusion}. 

Notation is standard and follows largely \cite{BC2011}. 
We write $U\oplus V$ to indicate that the terms of the Minkowski sum $U+V =
\menge{u+v}{u\in U,v\in V}$ satisfy $U\perp V$.

\section{Auxiliary results}

In this section, we collect various results to ease the derivation
of the main results.

\label{s:aux}

\subsection{Firmly nonexpansive mappings}

It is well known (see \cite{LM}, \cite{EckBer}, or \cite{JOSA}) 
that the Douglas--Rachford operator $T$ (see
\eqref{e:T}) is \emph{firmly nonexpansive}, i.e.,
\begin{equation}
(\forall x\in X)(\forall y\in X)
\quad
\|Tx-Ty\|^2 + \|(\Id-T)x-(\Id-T)y\|^2 \leq \|x-y\|^2.
\end{equation}
The following result will be useful in our analysis. 

\begin{fact}
\label{f:firmiter}
{\rm (See \cite[Corollary~5.16 and Proposition~5.27]{BC2011}, or
\cite[Theorem~2.2]{BDHP}, \cite{BaiBruRei} and \cite{BruRei}.)}
Let $T\colon X\to X$ be linear and firmly nonexpansive, and let
$x\in X$. 
Then $T^nx \to P_{\Fix T}x$. 
\end{fact}

\subsection{Products of projections and the Friedrichs angle}

Unless otherwise stated, we assume from now on that 
\begin{empheq}[box=\mybluebox]{equation}
\text{$U$ and $V$ are closed subspaces of $X$.}
\end{empheq}

The proof of the following useful fact 
can be found in \cite[Lemma~9.2]{Deutsch}:
\begin{equation}
\label{e:PP}
U\subseteq V 
\;\;\Rightarrow\;\;
P_U(V)\subseteq V
\;\;\Leftrightarrow\;\;
P_VP_U=P_UP_V = P_{U\cap V}.
\end{equation}

Our main results are formulated using the notion of the Friedrichs
angle between $U$ and $V$. Let us review the definition and
provide the key results which are needed in the sequel.

\begin{definition}
The \emph{cosine of the Friedrichs angle}
between $U$ and $V$ is 
\begin{equation}
c_F := \sup \menge{\scal{u}{v}}{u\in U\cap (U\cap
V)^\perp,\; v\in V\cap (U\cap V)^\perp,\;
\|u\|\leq 1,\; \|v\|\leq 1}. 
\end{equation}
We write $c_F(U,V)$ for $c_F$ if we emphasize the subspaces
utilized. 
\end{definition}

\begin{fact}[fundamental properties of the Friedrichs angle]
\label{f:angle}
Let $n\in\{1,2,3,\ldots\}$. Then the following hold:
\begin{enumerate}
\item 
\label{f:anglepos}
$U+V$ is closed
$\Leftrightarrow$ 
$c_F < 1$.
\item 
\label{f:angleSolmon}
$c_F(U,V) =c_F(V,U)= c_F(U^\perp,V^\perp)$.
\item
\label{f:angle3}
{$c_F=\|P_VP_U-P_{U\cap V}\|=\|P_{V^\perp}P_{U^\perp}-P_{U^\perp\cap V^\perp}\|$}.
\item
\label{f:angleAKW}
{\rm \textbf{(Aronszajn--Kayalar--Weinert)}}
$\|(P_VP_U)^n-P_{U\cap V}\| = c_F^{2n-1}$. 
\end{enumerate}
\end{fact}

\begin{proof}
\ref{f:anglepos}:
See \cite[Theorem~13]{Maratea}
\ref{f:angleSolmon}:
See \cite[Theorem~16]{Maratea}
\ref{f:angle3}:
{See \cite[Lemma~9.5(7)]{Deutsch} and \ref{f:angleSolmon} above. }
\ref{f:angleAKW}: 
See \cite[Theorem~9.31]{Deutsch} (or the original works
\cite{Aron} and \cite{KW}). 
\end{proof}

Note that Fact~\ref{f:angle}\ref{f:angleSolmon}\&\ref{f:angle3} yields
\begin{equation}
\label{e:cF0}
c_F = 0
\quad\Leftrightarrow\quad
P_VP_U = P_UP_V = P_{U\cap V}.
\end{equation}

The classical Fact~\ref{f:angle}\ref{f:angleAKW} deals with 
\emph{even} powers of alternating projectors. 
We complement this result by deriving the counterpart for
\emph{odd} powers. 

\begin{lemma}[odd powers of alternating projections]
\label{l:main} 
Let $n\in\{1,2,3,\ldots\}$. 
Then
\begin{equation}
\label{l:eq}
\|P_U(P_VP_U)^n-P_{U\cap V}\|=c_F^{2n}.
\end{equation}
\end{lemma}
\begin{proof}
If $c_F=0$, then the conclusion is clear from \eqref{e:cF0}.
We thus assume that $c_F>0$. 
By \eqref{e:PP}, 
\begin{subequations}
\begin{align}
\big((P_UP_V)^n-P_{U\cap V}\big)\big(P_VP_U-P_{U\cap V}\big)&=(P_UP_V)^nP_VP_U-(P_UP_V)^nP_{U\cap V}-P_{U\cap V}P_VP_U+P_{U\cap V}^2\\
&=(P_UP_V)^nP_U-P_{U\cap V}-P_{U\cap V}+P_{U\cap V}\\
&=P_U(P_VP_U)^n-P_{U\cap V}.
\end{align}
\end{subequations}
It thus follows from Fact~\ref{f:angle}\ref{f:angleAKW}\&\ref{f:angle3} that
\begin{equation}\label{lux1}
\|P_U(P_VP_U)^n-P_{U\cap V}\|\le \|(P_UP_V)^n-P_{U\cap V}\|\cdot\|P_VP_U-P_{U\cap V}\|=c_F^{2n}. 
\end{equation}
Since 
$(P_U(P_VP_U)^n-P_{U\cap V})(P_UP_V-P_{U\cap
V})=(P_UP_V)^{n+1}-P_{U\cap V}$,
we obtain from 
Fact~\ref{f:angle}\ref{f:angleAKW}\&\ref{f:angle3} that 
\begin{subequations}
\begin{align}
c_F^{2n+1}&=\|(P_UP_V)^{n+1}-P_{U\cap V}\|
\leq \|P_U(P_VP_U)^{n}-P_{U\cap V}\|\cdot\|P_UP_V-P_{U\cap
V}\|\\
&=\|P_U(P_VP_U)^n-P_{U\cap V}\|c_F. 
\end{align}
\end{subequations}
Since $c_F>0$, we obtain
$c_F^{2n} \leq \|P_U(P_VP_U)^n-P_{U\cap V}\|$.
Combining with \eqref{lux1}, we deduce \eqref{l:eq}. 
\end{proof}

\section{Basic properties of the Douglas--Rachford splitting operator}

\label{s:static}

We recall that the Douglas--Rachford operator is defined by 
\begin{empheq}[box=\mybluebox]{equation}
\label{e:T2}
T = T_{V,U} := P_V(2P_U-\Id) + \Id- P_U.
\end{empheq}

For future reference, we record the following consequence of
Fact~\ref{f:firmiter}.

\begin{corollary}
\label{c:jammery}
Let $x\in X$.
Then $T^nx \to P_{\Fix T}x$.
\end{corollary}

It will be very useful to work with reflectors, which we define
next.

\begin{definition}[reflector]
The \emph{reflector} associated with $U$ is 
\begin{empheq}[box=\mybluebox]{equation}
\label{e:defR}
R_U := 2P_U - \Id = P_U - P_{U^\perp}. 
\end{empheq}
\end{definition}

The following simple yet useful result is easily verified.

\begin{proposition}
\label{p:easyR}
The reflector $R_U$ is a surjective isometry with
\begin{equation}
R_U^* = R_U^{-1} = -R_{U^\perp} = R_U.
\end{equation}
\end{proposition}

We now record several reformulations of $T$ and
$T^*$ which we will use repeatedly in the paper without
explicitly mentioning it. 

\begin{proposition}
\label{p:otherT}
The following hold:
\begin{enumerate}
\item
\label{p:otherT1}
$T = \tfrac{1}{2}\Id + \tfrac{1}{2}R_VR_U = 
P_VR_U + \Id-P_U
= P_VP_U + P_{V^\perp}P_{U^\perp}$. 
\item
\label{p:otherT2}
$T^* = P_UP_V + P_{U^\perp}P_{V^\perp} = T_{U,V} = T^*_{V,U}$. 
\item
\label{p:otherT3}
$T_{V,U} = T_{V^\perp,U^\perp}$.
\end{enumerate}
\end{proposition}
\begin{proof}
\ref{p:otherT1}: Expand and use the linearity of $P_U$ and $P_V$.
\ref{p:otherT2}\&\ref{p:otherT3}: 
This follows from \ref{p:otherT1}. 
\end{proof}

The next result highlights the importance of the reflectors
when traversing between $T$ and $T^*$. 
Item~\ref{p:union3+} was observed previously in 
\cite[Remark~4.1]{BT13jota}. 

\begin{proposition}
\label{p:union}
The following hold:
\begin{enumerate}
\item
\label{p:union1}
$R_UT^*=TR_U = T^*R_V=R_VT=P_U+P_V-\Id$.
\item
\label{p:union2}
$T^*(R_VR_U)=(R_VR_U)T^*=T$ and
$T(R_UR_V)=(R_UR_V)T=T^*$.
\item
\label{p:union3}
$TT^*=T^*T$, i.e., $T$ is normal. 
\item
\label{p:union3+}
$2TT^* = T+T^*$. 
\item
\label{p:union3++}
$TT^*$ is firmly nonexpansive and self-adjoint. 
\item
\label{p:union4}
$TT^* = 
P_VP_UP_V + P_{V^\perp}P_{U^\perp}P_{V^\perp}
= 
P_VP_U+P_UP_V-P_U-P_V+\Id 
= 
P_UP_VP_U + P_{U^\perp}P_{V^\perp}P_{U^\perp}$.
\end{enumerate}
\end{proposition}
\begin{proof}
\ref{p:union1}: 
Indeed, 
using Proposition~\ref{p:otherT}\ref{p:otherT1}\&\ref{p:otherT2}, 
we see that 
\begin{subequations}
\begin{align}
TR_U &=(P_VP_U+P_{V^\perp}P_{U^\perp})(P_U-P_{U^\perp})
= P_VP_U-P_{V^\perp}P_{U^\perp}
= P_VP_U -(\Id-P_V)(\Id-P_U)\\
&= P_U+P_V-\Id\\
&= P_UP_V -(\Id-P_U)(\Id-P_V)
= P_UP_V - P_{U^\perp}P_{V^\perp}
= (P_UP_V+P_{U^\perp}P_{V^\perp})(P_V-P_{V^\perp})\\
&= T^*R_V
\end{align}
\end{subequations}
is \emph{self-adjoint}. 
Hence 
$R_UT^* = (TR_U)^* = TR_U = T^*R_V = (R_VT)^* = R_VT$ 
by Proposition~\ref{p:easyR}. 

\ref{p:union2}: Clear from \ref{p:union1}. 

\ref{p:union3}: Using \ref{p:union2},
we obtain 
$TT^* = T^*(R_VR_U)(R_UR_V)T = T^*T$.

\ref{p:union3+}: 
$4TT^* = (\Id+R_VR_U)(\Id+R_UR_V) = 2\Id+R_VR_U+R_UR_V
=2( (\Id+R_VR_U)/2 + (\Id+R_UR_V)/2) = 2(T+T^*)$. 

\ref{p:union3++}: 
Since $T$ and $T^*$ are firmly nonexpansive, so is their convex combination
$(T+T^*)/2$, which equals $TT^*$ by \ref{p:union3+}. 
It is clear that $TT^*$ is self-adjoint. 

\ref{p:union4}: 
It follows 
from Proposition~\ref{p:otherT}\ref{p:otherT1}\&\ref{p:otherT2} that 
$TT^* =
(P_VP_U+P_{V^\perp}P_{U^\perp})(P_UP_V+P_{U^\perp}P_{V^\perp}) = P_VP_UP_V
+ P_{V^\perp}P_{U^\perp}P_{V^\perp}$, which yields the first equality. 
Replacing $P_{U^\perp}$ and $P_{V^\perp}$ by
$\Id-P_U$ and $\Id-P_V$, respectively, followed by expanding and
simplifying results in the second equality. The last equality is  
proved analogously. 
\end{proof}

Parts of our next result were also obtained in \cite{HL} and
\cite{HLN} when $X$ is finite-dimensional. 

\begin{proposition}
\label{p:Fix}
Let $\nnn$. 
Then the following hold:
\begin{enumerate}
\item
\label{p:Fix1}
$\Fix T = \Fix T^* = \Fix T^*T = (U\cap V)\oplus(U^\perp\cap V^\perp)$.
\item
\label{p:Fix1+}
$\Fix T = U\cap V$ 
$\Leftrightarrow$
$\overline{U+V}=X$. 
\item
\label{p:Fix2}
$P_{\Fix T} = P_{U\cap V}+ P_{U^\perp\cap V^\perp}$. 
\item \label{p:Fix4} 
$P_{\Fix T}T=TP_{\Fix T}=P_{\Fix T}$.
\item 
\label{p:Fix3}
$P_UP_{\Fix T} = P_VP_{\Fix T}= P_{U\cap V}P_{\Fix T} = P_{U\cap V}
= P_{U\cap V}P_{\Fix T}T^n = P_{U\cap V}T^n$. 
\end{enumerate}
\end{proposition}
\begin{proof}
\ref{p:Fix1}:
Set $A = N_U$ and $B=N_V$. Then $(A+B)^{-1}(0) = U\cap V$. 
Combining \cite[Example~2.7 and Corollary~5.5(iii)]{BBHM} yields
$\Fix T = (U\cap V)\oplus(U^\perp\cap V^\perp)$. 
By \cite[Lemma~2.1]{BDHP}, we have $\Fix T = \Fix T^*$. 
Since $T$ and $T^*$ are firmly nonexpansive, and $0\in \Fix T\cap \Fix
T^*$, we apply 
\cite[Corollary~4.3 and Corollary~4.37]{BC2011} 
to deduce that 
$\Fix T^*T = \Fix T \cap \Fix T^*$. 

\ref{p:Fix1+}:
Using \ref{p:Fix1}, we obtain $\Fix T=U\cap V$
$\Leftrightarrow$
$\Up\cap\Vp=\{0\}$
$\Leftrightarrow$
$\overline{U+V} = \overline{U^{\perp\perp} + V^{\perp\perp}} =
(\Up\cap\Vp)^\perp =\{0\}^\perp = X$.

\ref{p:Fix2}: 
This follows from \ref{p:Fix1}. 

\ref{p:Fix4}: Clearly, $TP_{\Fix T}=P_{\Fix T}$. Furthermore, 
$P_{\Fix T}T= TP_{\Fix T}$ by \cite[Lemma~3.12]{BDHP}. 

\ref{p:Fix3}: 
First, \ref{p:Fix2} and \eqref{e:PP} imply 
$P_UP_{\Fix T} = P_VP_{\Fix T}= P_{U\cap V}P_{\Fix T} = P_{U\cap V}$. 
This and \ref{p:Fix4} give the remaining equalities. 
\end{proof}

\section{Main result}
\label{s:main}

We now are ready for our main results concerning the dynamical
behaviour of the Douglas--Rachford iteration.

\begin{theorem}[powers of $T$]
\label{t:main}
Let $n\in\{1,2,3,\ldots\}$, and let $x\in X$. Then 
\begin{subequations}
\begin{align}
\label{mainline}
\|T^n-P_{\Fix T}\|&= c_F^n, \\
\|(TT^*)^n-P_{\Fix T}\| &= c_F^{2n}, 
\label{mainline3}
\end{align}
\end{subequations}
and 
\begin{equation}
\label{mainline2}
\|T^nx-P_{\Fix T}x\|\le c_F^n\|x-P_{\Fix T}x\|\le c_F^n\|x\|.
\end{equation}
\end{theorem}
\begin{proof}
Set 
\begin{equation}
c := \|T-P_{\Fix T}\| = \|TP_{(\Fix T)^\perp}\|,
\end{equation}
and observe that the second equality is justified 
since $P_{\Fix T} = TP_{\Fix T}$.
Since $T$ is (firmly) nonexpansive and normal 
(see Proposition~\ref{p:union}\ref{p:union3}), 
it follows from \cite[Lemma~3.15(i)]{BDHP} that
\begin{equation}\label{e:c(T)}
\|T^n-P_{\Fix T}\|=c^n. 
\end{equation}
Proposition~\ref{p:otherT}\ref{p:otherT1},
Proposition~\ref{p:Fix}\ref{p:Fix2}, 
and \eqref{e:PP} imply 
\begin{subequations}
\begin{align}
(T-P_{\Fix T})x&=\big(P_VP_U+P_{V^\perp}P_{U^\perp}\big)x-\big(P_{U\cap V}+P_{U^\perp\cap V^\perp}\big)x\\
&=\underbrace{\big(P_VP_U-P_{U\cap V}\big)x}_{\in
V}+\underbrace{\big(P_{V^\perp}P_{U^\perp}-P_{U^\perp\cap
V^\perp}\big)x}_{\in V^\perp}\label{24b}\\
&=\big(P_VP_UP_Ux-P_{U\cap V}P_Ux\big)+\big(P_{V^\perp}P_{U^\perp}P_{U^\perp}x-P_{U^\perp\cap V^\perp}P_{U^\perp}x\big)\\
&=\underbrace{\big(P_VP_U-P_{U\cap V}\big)P_Ux}_{\in V} + 
\underbrace{\big(P_{V^\perp}P_{U^\perp}-P_{U^\perp\cap
V^\perp}\big)P_{U^\perp}x}_{\in V^\perp}.
\label{e:24d}
\end{align}
\end{subequations}
Hence, using \eqref{e:24d} and Fact~\ref{f:angle}\ref{f:angle3}, we obtain
\begin{subequations}
\begin{align}
\|(T-P_{\Fix T})x\|^2&=\big\|\big(P_VP_U-P_{U\cap V}\big)P_Ux\big\|^2+\big\|\big(P_{V^\perp}P_{U^\perp}-P_{U^\perp\cap V^\perp}\big)P_{U^\perp}x\big\|^2\\
&\le\big\|P_VP_U-P_{U\cap V}\big\|^2 \|P_Ux\|^2+\big\|P_{V^\perp}P_{U^\perp}-P_{U^\perp\cap V^\perp}\big\|^2\|P_{U^\perp}x\|^2\\
&= c_F^2 \|P_Ux\|^2+c_F^2\|P_{U^\perp}x\|^2\\
&=c_F^2\|x\|^2.
\end{align}
\end{subequations}
We deduce that 
\begin{equation}
\label{e:1206a}
c=\|T-P_{\Fix T}\|\leq c_F. 
\end{equation}
Furthermore,  \eqref{24b} implies
\begin{subequations}
\begin{align}
\|(T-P_{\Fix T})x\|^2&=
\big\|\big(P_VP_U-P_{U\cap V}\big)x\big\|^2+\big\|\big(P_{V^\perp}P_{U^\perp}-P_{U^\perp\cap V^\perp}\big)x\big\|^2\\
&\ge \big\|\big(P_VP_U-P_{U\cap V}\big)x\big\|^2.
\end{align}
\end{subequations}
This and Fact~\ref{f:angle}\ref{f:angle3} yield
\begin{equation}
\label{e:1206b}
c=\|T-P_{\Fix T}\|\geq \|P_VP_U-P_{U\cap V}\|=c_F.
\end{equation}
Combining \eqref{e:1206a} and \eqref{e:1206b}, 
we obtain $c=c_F$. 
Consequently, \eqref{e:c(T)} yields \eqref{mainline}
while \eqref{mainline3} follows from \cite[Lemma~3.15.(2)]{BDHP}. 

Finally, \cite[Lemma~3.14(1)\&(3)]{BDHP} results in 
$\|T^nx-P_{\Fix T}x\|\le c^n\|x-P_{\Fix T}x\|=c_F^n\|x-P_{\Fix
T}x\|=c_F^n\|P_{{\Fix T}^\perp}x\|\le c_F^n\|x\|$.
\end{proof}

The following result yields further insights in various powers of $T$ and
$T^*$.

\begin{proposition}
\label{p:laura}
Let $\nnn$. Then the following hold:
\begin{enumerate}
\item
\label{p:laura1}
$(TT^*)^n = (T^*T)^n = (P_UP_VP_U)^n + (P_{U^\perp}P_{V^\perp}P_{U^\perp})^n 
= (P_VP_UP_V)^n + (P_{V^\perp}P_{U^\perp}P_{V^\perp})^n$ if $n\geq 1$. 
\item 
\label{p:laura2}
$P_U(TT^*)^n=(P_UP_V)^nP_U = (TT^*)^nP_U$ 
and $P_V(TT^*)^n=(P_VP_U)^nP_V=(P_VP_U)^nP_V$.
\item  
\label{p:laura4}
$T^{2n} = (TT^*)^n(R_VR_U)^n$.
\item  
\label{p:laura5}
$T^{2n+1} = (TT^*)^nT(R_VR_U)^n = (TT^*)^nT^*(R_VR_U)^{n+1}$.
\item  
\label{p:laura6}
$T^{2n}=
\big((P_UP_V)^nP_U+(P_{U^\perp}P_{V^\perp})^nP_{U^\perp}\big)(R_VR_U)^{n}
=\big((P_VP_U)^nP_V+(P_{V^\perp}P_{U^\perp})^nP_{V^\perp}\big)(R_VR_U)^{n}$.
\item  
\label{p:laura7}
$T^{2n+1}= 
\big((P_UP_V)^{n+1}
+ (P_{U^\perp}P_{V^\perp})^{n+1}\big)(R_VR_U)^{n+1}
= \big((P_VP_U)^{n+1} +
(P_{V^\perp}P_{U^\perp})^{n+1}\big)(R_VR_U)^{n}$.
\end{enumerate}
\end{proposition}
\begin{proof}
\ref{p:laura1}: 
Proposition~\ref{p:union}\ref{p:union3}\&\ref{p:union4} yield
\begin{equation}
(TT^*)^n = (P_UP_VP_U + P_{U^\perp}P_{V^\perp}P_{U^\perp})^n = 
(P_UP_VP_U)^n + (P_{U^\perp}P_{V^\perp}P_{U^\perp})^n;
\end{equation}
the last equality follows similarly. 

\ref{p:laura2}:
By \ref{p:laura1}, 
$P_U(TT^*)^n = (TT^*)^nP_U =(P_UP_VP_U)^n = (P_UP_V)^nP_U$. 
The proof of the remaining equalities is similar. 

\ref{p:laura4}: 
Since $T = T^*R_VR_U = R_VR_UT^*$ 
(see Proposition~\ref{p:union}\ref{p:union2}), 
we have $T^n = (T^*)^n(R_VR_U)^n$.
Thus 
$T^{2n} = T^n(T^*)^n(R_VR_U)^n$, and therefore
$T^{2n} = (TT^*)^n(R_VR_U)^n$ using
Proposition~\ref{p:union}\ref{p:union3}. 

\ref{p:laura5}: 
Using \ref{p:laura4} and
Proposition~\ref{p:union}\ref{p:union3}\&\ref{p:union2}, 
we have
$T^{2n+1} = TT^{2n} = T(TT^*)^n(R_VR_U)^n=
(TT^*)^nT(R_VR_U)^n = (TT^*)^{n}(T^*R_VR_U)(R_VR_U)^n
= (TT^*)^{n}T^*(R_VR_U)^{n+1}$.

\ref{p:laura6}\&\ref{p:laura7}:
Using \ref{p:laura4}, \ref{p:laura1}, \ref{p:laura5}, 
and Proposition~\ref{p:otherT}\ref{p:otherT2}, 
we have 
\begin{equation}
T^{2n} = 
\big((P_UP_VP_U)^n+(P_{U^\perp}P_{V^\perp}P_{U^\perp})^n\big)(R_VR_U)^{n}
= \big((P_UP_V)^nP_U+(P_{U^\perp}P_{V^\perp})^nP_{U^\perp}\big)(R_VR_U)^{n}
\end{equation}
and 
\begin{subequations}
\begin{align}
T^{2n+1} &= 
\big((P_UP_VP_U)^n+(P_{U^\perp}P_{V^\perp}P_{U^\perp})^n\big)
(P_UP_V+P_{U^\perp}P_{V^\perp})(R_VR_U)^{n+1}\\
&=\big((P_UP_V)^{n+1}
+ (P_{U^\perp}P_{V^\perp})^{n+1}\big)(R_VR_U)^{n+1}.
\end{align}
\end{subequations}
Similarly, using
\ref{p:laura4},  \ref{p:laura1},  \ref{p:laura5}, 
and Proposition~\ref{p:otherT}\ref{p:otherT2}, 
we have 
\begin{subequations}
\begin{align}
T^{2n} 
&= \big((P_VP_UP_V)^n+(P_{V^\perp}P_{U^\perp}P_{V^\perp})^n\big)(R_VR_U)^{n}\\
&= \big((P_VP_U)^nP_V+(P_{V^\perp}P_{U^\perp})^nP_{V^\perp}\big)(R_VR_U)^{n}
\end{align}
\end{subequations}
and 
\begin{subequations}
\begin{align}
T^{2n+1} &=
\big((P_VP_UP_V)^n+(P_{V^\perp}P_{U^\perp}P_{V^\perp})^n\big)(P_VP_U+P_{V^\perp}P_{U^\perp})(R_VR_U)^{n}\\
&= \big((P_VP_U)^{n+1}
+ (P_{V^\perp}P_{U^\perp})^{n+1}\big)(R_VR_U)^{n}.
\end{align}
\end{subequations}
The proof is complete.
\end{proof}

We are now ready for our main result.
Note that item~\ref{t:main3} is the counterpart of
Fact~\ref{f:angle}\ref{f:angleAKW} for the Douglas--Rachford algorithm.

\begin{theorem}[shadow powers of $T$]
\label{t:main2}
Let $\nnn$, and let $x\in X$. 
Then the following hold:
\begin{enumerate}
\item
\label{t:main3}
$\|P_UT^n - P_{U\cap V}\|=\|P_VT^n - P_{U\cap V}\|= c_F^n$.
\item
\label{t:main3+}
$\max\big\{\|P_UT^nx-P_{U\cap V}x\|,\|P_VT^nx-P_{U\cap V}x\|\big\}\leq
c_F^n\|x\|$.
\item
\label{t:main4}
$\|P_VT^{n+1}x - P_{U\cap V}x\| \leq 
c_F\|P_UT^nx-P_{U\cap V}x\|\le  c_F\|T^nx-P_{\Fix T}x\|\le
c_F^{n+1}\|x\|$. 
\end{enumerate}
\end{theorem}
\begin{proof} 
\ref{t:main3}:
Note that $P_{U\cap V}R_VR_U =
P_{U\cap V}(P_V-P_{V^\perp})R_U
=P_{U\cap V}R_U = P_{U\cap V}(P_U-P_{U^\perp})
= P_{U\cap V}$. 
It follows that $P_{U\cap V}=P_{U\cap V}(R_VR_U)^n = P_{U\cap
V}(R_VR_U)^{n+1}$. 
Hence, using Proposition~\ref{p:laura}\ref{p:laura6}, we have
$P_UT^{2n}-P_{U\cap V} = (P_UP_V)^nP_U(R_VR_U)^n - P_{U\cap V}(R_VR_U)^n 
= ((P_UP_V)^nP_U-P_{U\cap V})(R_VR_U)^n$
and thus
$\|P_UT^{2n}-P_{U\cap V}\| = c_F^{2n}$ by Lemma~\ref{l:main}. 
It follows likewise from Proposition~\ref{p:laura}\ref{p:laura7} and 
Fact~\ref{f:angle}\ref{f:angleAKW} that 
$\|P_UT^{2n+1}-P_{U\cap V}\|=\|((P_UP_V)^{n+1}-P_{U\cap
V})(R_VR_U)^{n+1}\|=\|(P_UP_V)^{n+1}-P_{U\cap V}\|=c_F^{2n+1}$.
Thus, we have $\|P_UT^{n}-P_{U\cap V}\|=c_F^n$. 
The proof of $\|P_VT^{n}-P_{U\cap V}\|=c_F^n$ is analogous.

\ref{t:main3+}: Clear from \ref{t:main3}. 

\ref{t:main4}:
Using Proposition~\ref{p:otherT}\ref{p:otherT1} and
Proposition~\ref{p:Fix}\ref{p:Fix3}, 
we have 
\begin{subequations}
\begin{align}
P_VT^{n+1}-P_{U\cap V} &=
P_V(P_VP_U+P_{V^\perp}P_{U^\perp})T^n-P_{U\cap V}P_{\Fix T}
= P_VP_UT^n - P_{U\cap V}P_{\Fix T}T^n\\
&= P_VP_UT^n - P_{U\cap V} + P_{U\cap V} - P_{U\cap V}T^n
=(P_VP_U-P_{U\cap V})(P_UT^n-P_{U\cap V}).
\end{align}
\end{subequations}
Combining this with Fact~\ref{f:angle}\ref{f:angle3} and
\eqref{mainline}, we get
$\|P_VT^{n+1}x-P_{U\cap V}x\| \leq c_F\|P_UT^{n}x-P_{U\cap V}x\|
=c_F\|P_U(T^{n}x-P_{\Fix T}x)\|
\leq c_F\|T^{n}x-P_{\Fix T}x\| \leq c_F^{n+1}\|x\|$.
\end{proof}

\begin{corollary}[linear convergence]
\label{c:main2}
We have
$T^nx\to P_{(U\cap V)+(U^\perp \cap V^\perp)}x$,
$P_UT^nx \to P_{U\cap V}x$, and
$P_VT^nx \to P_{U\cap V}x$.
If $U+V$ is closed, 
then convergence of these sequences is linear with rate $c_F<1$. 
\end{corollary}
\begin{proof}
Corollary~\ref{c:jammery}, 
Proposition~\ref{p:Fix}\ref{p:Fix2}\&\ref{p:Fix3} and
\eqref{e:PP} imply
$P_UT^nx\to P_UP_{\Fix T}x= P_{U\cap V}x$ and analogously
$P_VT^nx\to P_{U\cap V}x$.
Recall from 
Fact~\ref{f:angle}\ref{f:anglepos}
that $U+V$ is closed if and only if $c_F<1$.
The conclusion is thus clear from Theorem~\ref{t:main2}\ref{t:main4}. 
\end{proof}

A translation argument gives the following result (see also
\cite[Theorem~3.17]{BCL04} for an earlier related result).

\begin{corollary}[affine subspaces]
Suppose that $U$ and $V$ are closed \emph{affine} subspaces
of $X$ such that $U\cap V\neq\varnothing$, and let $x\in X$. 
Then 
\begin{equation}
T^nx \to P_{\Fix T}x, \quad 
P_UT^nx \to P_{U\cap V}x, \quad\text{and}\quad P_VT^nx \to P_{U\cap V}x. 
\end{equation}
If $(U-U)+(V-V)$ is closed,
then the convergence is linear with rate $c_F(U-U,V-V)<1$. 
\end{corollary}

\section{Two lines in the Euclidean plane}

\label{s:plane}

We present some geometric results concerning the lines in the
plane which will not only be useful later but which also illustrate
the results of the previous sections.
In this section, we assume that $X=\RR^2$, and we set
\begin{equation}
e_0 := (1,0),\quad
e_{\pi/2} := (0,1), \quad
\text{and}\quad
\big(\forall \theta\in[0,\pi/2]\big)\;\;
e_\theta := \cos(\theta)e_0 + \sin(\theta)e_{\pi/2}.
\end{equation}
Define the (counter-clockwise) rotator by
\begin{equation}
\big(\forall \theta\in\RP)\quad
R_\theta := \begin{pmatrix}
\cos(\theta) & -\sin(\theta)\\
\sin(\theta) & \cos(\theta)
\end{pmatrix},
\end{equation}
and note that $R_\theta^{-1} = R_\theta^*$.
Now let $\theta\in\left]0,\pi/2\right]$, 
and suppose that
\begin{equation}
U=\RR\cdot e_0
\quad\text{and}\quad
V=\RR\cdot e_\theta = R_\theta(U). 
\end{equation}
Then
\begin{equation}
U\cap V = \{0\}
\quad\text{and}\quad
c_F(U,V) = \cos(\theta).
\end{equation}
By, e.g., \cite[Proposition~28.2(ii)]{BC2011},
$P_V = R_\theta P_UR_\theta^*$.
In terms of matrices, we thus have
\begin{equation}
P_U =
\begin{pmatrix}
1 & 0 \\
0 & 0
\end{pmatrix}
\quad\text{and}\quad
P_V =
\begin{pmatrix}
\cos^2(\theta) & \sin(\theta)\cos(\theta)\\
\sin(\theta)\cos(\theta) & \sin^2(\theta)
\end{pmatrix}.
\end{equation}
Consequently, 
the corresponding Douglas--Rachford splitting operator 
is 
\begin{subequations}
\begin{align}
T &= P_V(2P_U-\Id) + \Id - P_U =
\begin{pmatrix}
\cos^2(\theta) & -\sin(\theta)\cos(\theta)\\
\sin(\theta)\cos(\theta) & 1-\sin^2(\theta)
\end{pmatrix}\\
&=
\cos(\theta)\begin{pmatrix}
\cos(\theta) & -\sin(\theta)\\
\sin(\theta)& \cos(\theta)
\end{pmatrix}
=\cos(\theta) R_\theta.
\end{align}
\end{subequations}
Thus,
$\Fix T = \{0\} = P_U(\Fix T)$,
\begin{equation}
(\forall\nnn)\quad 
T^n = \cos^n(\theta)R_{n\theta},
\quad
P_UT^n = \cos^n(\theta)\begin{pmatrix}
\cos(n\theta) & -\sin(n\theta)\\
0 & 0 
\end{pmatrix}, 
\end{equation}
and
\begin{equation}\label{e:16}
(\forall x\in X)\quad 
\|T^n x\| = \cos^n(\theta)\|x\|,
\;\;
\|P_UT^nx\| = \cos^n(\theta)|\cos(n\theta)x_1-\sin(n\theta)x_2|. 
\end{equation}
Furthermore,
\begin{equation}
(\forall n\geq 1)\quad 
(P_VP_U)^n =
\cos^{2n-1}(\theta)\begin{pmatrix}
\cos(\theta) & 0 \\
\sin(\theta) & 0 
\end{pmatrix}
\end{equation}
and thus
\begin{equation}\label{e:18}
(\forall x\in X)\quad
\|(P_VP_U)^nx\| = \cos^{2n-1}(\theta)|x_1|\leq
\cos^{2n-1}(\theta)\|x\|. 
\end{equation}

\section{An example without linear rate of convergence}

\label{s:anex}

In this section, let us assume that our underlying Hilbert space is 
$\ell^2(\NN) = \RR^2\oplus \RR^2 \oplus \cdots$. 
It will be more suggestive to use boldface letters
for vectors lying in, and operators acting on, $\ell^2(\NN)$. 
Thus, 
\begin{equation}
\bX = \ell^2(\NN),
\end{equation}
and we write is $\bx = (x_n)_\nnn =
((x_0,x_1),(x_2,x_3),\ldots)$ for a generic vector in $\bX$. 
Suppose that $(\theta_n)_\nnn$ is a sequence of angles in
$\left]0,\pi/2\right[$ with $\theta_n\to 0^+$. 
We set $(\forall\nnn)$ $c_n := \cos(\theta_n)\to 1^-$. 
We will use notation and results 
from Section~\ref{s:plane}. 
We assume that 
\begin{equation}\label{UU}
\bU = \RR\cdot e_0\times \RR\cdot e_0 \times \cdots \subseteq \bX
\end{equation}
and that
\begin{equation}\label{VV}
\bV = \RR\cdot e_{\theta_0}\times \RR\cdot
e_{\theta_1}\times\cdots \subseteq \bX.
\end{equation}
Then 
\begin{equation}
\bU\cap\bV = \{\boldsymbol{0}\}
\quad\text{and}\quad
c_F(\bU,\bV) = \sup_\nnn c_F(\RR\cdot e_0,\RR\cdot e_{\theta_n}) = \sup_\nnn
c_n = 1.
\end{equation}
The Douglas--Rachford splitting operator is
\begin{equation}
\bT = c_0R_{\theta_0}\oplus c_1R_{\theta_1}\oplus\cdots.
\end{equation}
Now let $\bx=(x_n)_\nnn \in\bX$ and $\gamma\in\zeroun$. 
Assume further that $\menge{\nnn}{x_n\neq 0}$ is infinite. 
Then there exists $N\in\NN$ such that
$(x_{2N},x_{2N+1})\in\RR^2\smallsetminus\{(0,0)\}$ and
$c_N>\gamma$. 
Hence
\begin{equation}
\gamma^{-n}\|\bT^n\bx\| \geq 
\gamma^{-n}c_N^n\|(x_{2N},x_{2N+1})\|\to\pinf;
\end{equation}
consequently, 
\begin{equation}
\text{$\bT^n\bx\to \boldsymbol{0}$,\, but not linearly with
rate $\gamma$.}
\end{equation}

Let us now assume in addition that 
$\theta_0=\pi/3$ and $(\forall n\geq 1)$ $\theta_n=\pi/(4n)$ and
$x_n = 1/(n+1)$. 
Then there exists $\delta>1$ and
$N\geq 1$ such that $(\forall n\geq N)$ $\gamma^{-1}c_n\geq
\delta$. Hence, for every $n\geq N$, we have
\begin{equation}
\gamma^{-n}\|P_{\bU}\bT^n\bx\|
\geq \gamma^{-n}c_n^n
|\cos(n\theta_n)x_{2n}-\sin(n\theta_n)x_{2n+1}|
\geq \frac{\delta^n}{2^{3/2}(n+1)(2n+1)}\to\pinf; 
\end{equation}
thus, 
\begin{equation}
\text{$P_{\bU}\bT^n\bx\to \boldsymbol{0}$,\, but not linearly with
rate $\gamma$.}
\end{equation}

In summary, these constructions illustrate that 
when the Friedrichs angle is zero, then
one cannot expect linear convergence of the
(projected) iterates of the Douglas--Rachford splitting operator. 

\section{Comparison with the method of alternating projections}

\label{s:compare}

Let us now compare our main results (Theorems~\ref{t:main} and
\ref{t:main2})
with the \emph{method of alternating projections}, for which
the following fundamental result is well known.

\begin{fact}[Aronszajn]
\label{f:map}
{\rm (See \cite{Aron} or \cite[Theorem~9.8]{Deutsch}.)} 
Let $x\in X$.
Then 
\begin{equation}
(\forall n\geq 1)\quad
\|(P_VP_U)^nx-P_{U\cap V}x\| \leq
c_F^{2n-1}\|x-P_{U\cap V}x\|. 
\end{equation}
\end{fact}

In Fact~\ref{f:map}, the rate $c_F$ is best possible (see 
Fact~\ref{f:angle}\ref{f:angleAKW} and the
results and comments in \cite[Chapter~9]{Deutsch}), 
and if the Friedrichs
angle is 0, then slow convergence may occur 
(see, e.g., \cite{BDH}). 

From Theorem~\ref{t:main}, Corollary~\ref{c:main2}, and Fact~\ref{f:map}, we see that 
the rate of convergence $c_F$ of $(T^nx)_\nnn$ to $P_{\Fix T}x$ and, 
\emph{a fortiori}, of $(P_UT^nx)_\nnn$ to $P_{U\cap V}x$, 
is clearly \emph{slower} than the rate of convergence $c_F^2$ of
$((P_VP_U)^nx)_\nnn$ to $P_{U\cap V}x$. 
In other words, the Douglas--Rachford splitting method 
appears to be twice as slow as the method
of alternating projections. 
While this is certainly the case for the iterates $(T^nx)_\nnn$,
the actual iterates of interest, namely $(P_UT^nx)_\nnn$, in
practice often (somewhat paradoxically) make striking
non-monotone progress. 

Let us illustrate this using the set up of Section~\ref{s:plane}
the notation and results of which we will utilize. 

Consider first \eqref{e:16} and \eqref{e:18} with
$x=e_0$ and $\theta={\pi}/{17}$.
In Figure~\ref{f1}, we show the first 100 iterates of
the sequences 
$(\|T^n x\|)_\nnn$ (red line), 
$(\|P_UT^nx\|)_\nnn$ (blue line), and $(\|(P_VP_U)^nx\|)_\nnn$
(green line). The sequences $(\|T^nx\|)_\nnn$ and
$(\|(P_VP_U)^nx\|)_\nnn$, which are decreasing,
represent the distance of the iterates to
$0$, the unique solution of the problem. 
While $(\|(P_VP_U)^nx\|)_\nnn$ decreases
faster than  $(\|T^nx\|)_\nnn$, the sequence of ``shadows'' 
$(\|P_UT^nx\|)_\nnn$ exhibits a curious non-monotone
``rippling'' behaviour
--- it may be quite close to the solution soon after 
the iteration starts!

\begin{figure}[H]
\begin{center}
\includegraphics[height=3.0in]{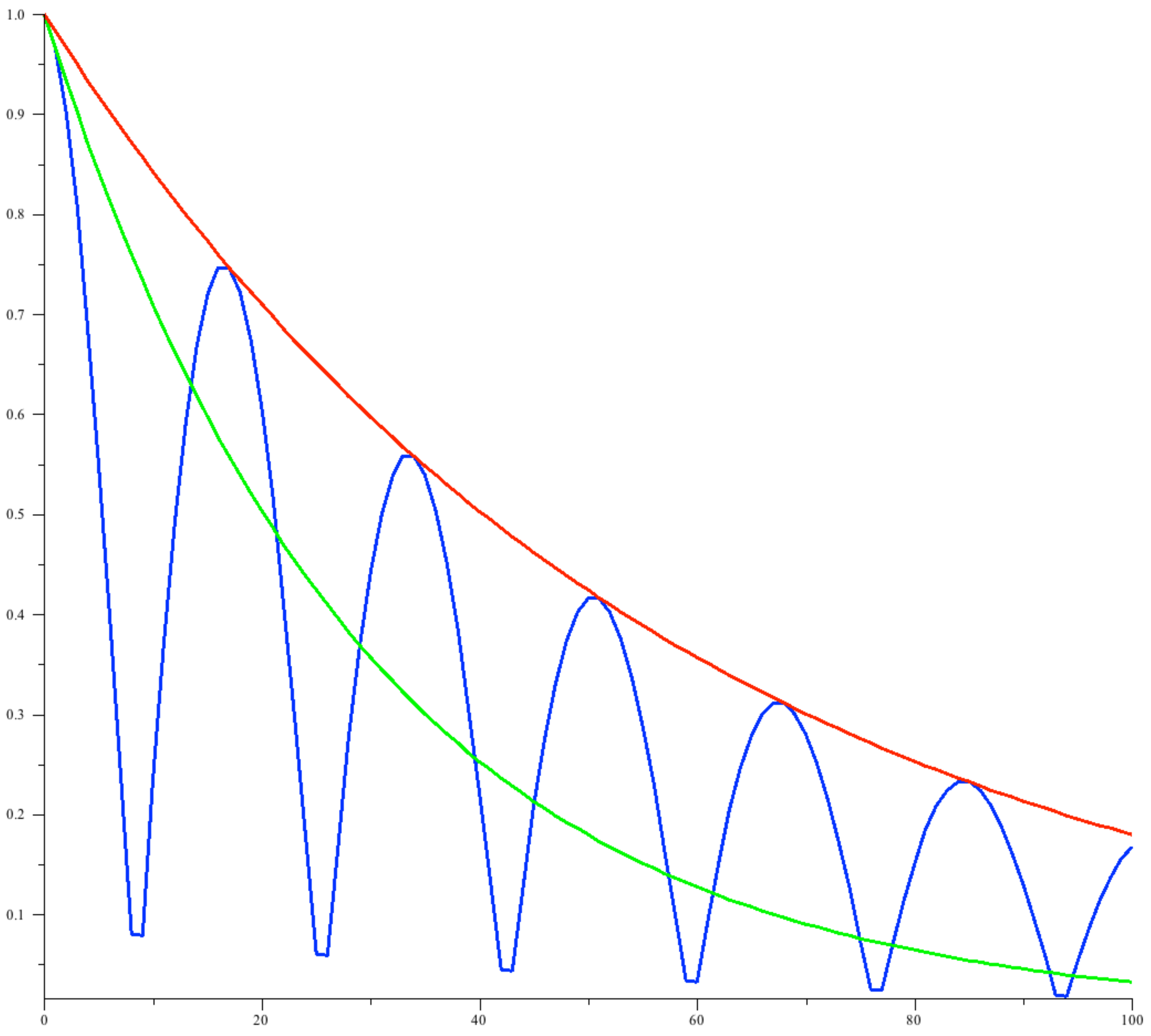}
\end{center}
\caption{
The distance of the first 100 terms of the sequences
$(T^nx)_\nnn$ (red), $(P_UT^nx)_\nnn$ (blue), and
$((P_VP_U)^nx)_\nnn$ (green) to the
unique solution.
\label{f1}}
\end{figure}

We next show in 
Figure~\ref{f2a} and Figure~\ref{f2b}
the first $100$ terms of 
$(\|T^nx\|)_\nnn$ and $(\|(P_VP_U)^nx\|)_\nnn$, where 
$\theta\colon[0,1]\to[0,\pi/2]\colon t\mapsto (\pi/2)t^3$ is 
parametrized to exhibit more clearly the behaviour for small
angles. Clearly, and as predicted, smaller angles correspond to
slower rates of convergence.

\begin{figure}[H]
\begin{center}
\includegraphics[height=3.0in]{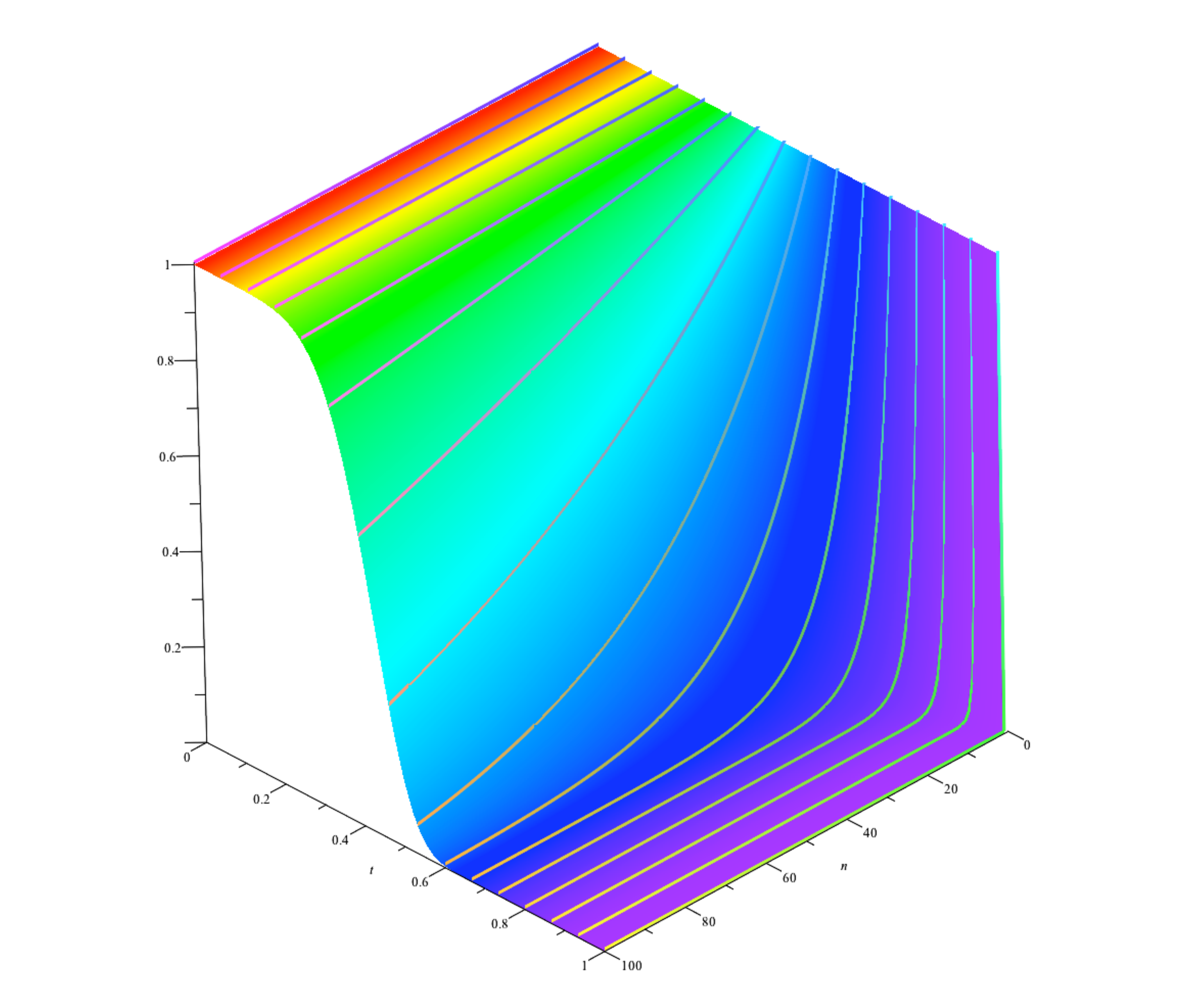}
\caption{
The distance of the first 100 terms of the sequence
$(T^nx)_\nnn$ to the unique solution when the angle ranges
between $0$ and $\pi/2$. 
\label{f2a}
}
\end{center}
\end{figure}

\begin{figure}[H]
\begin{center}
\includegraphics[height=3.0in]{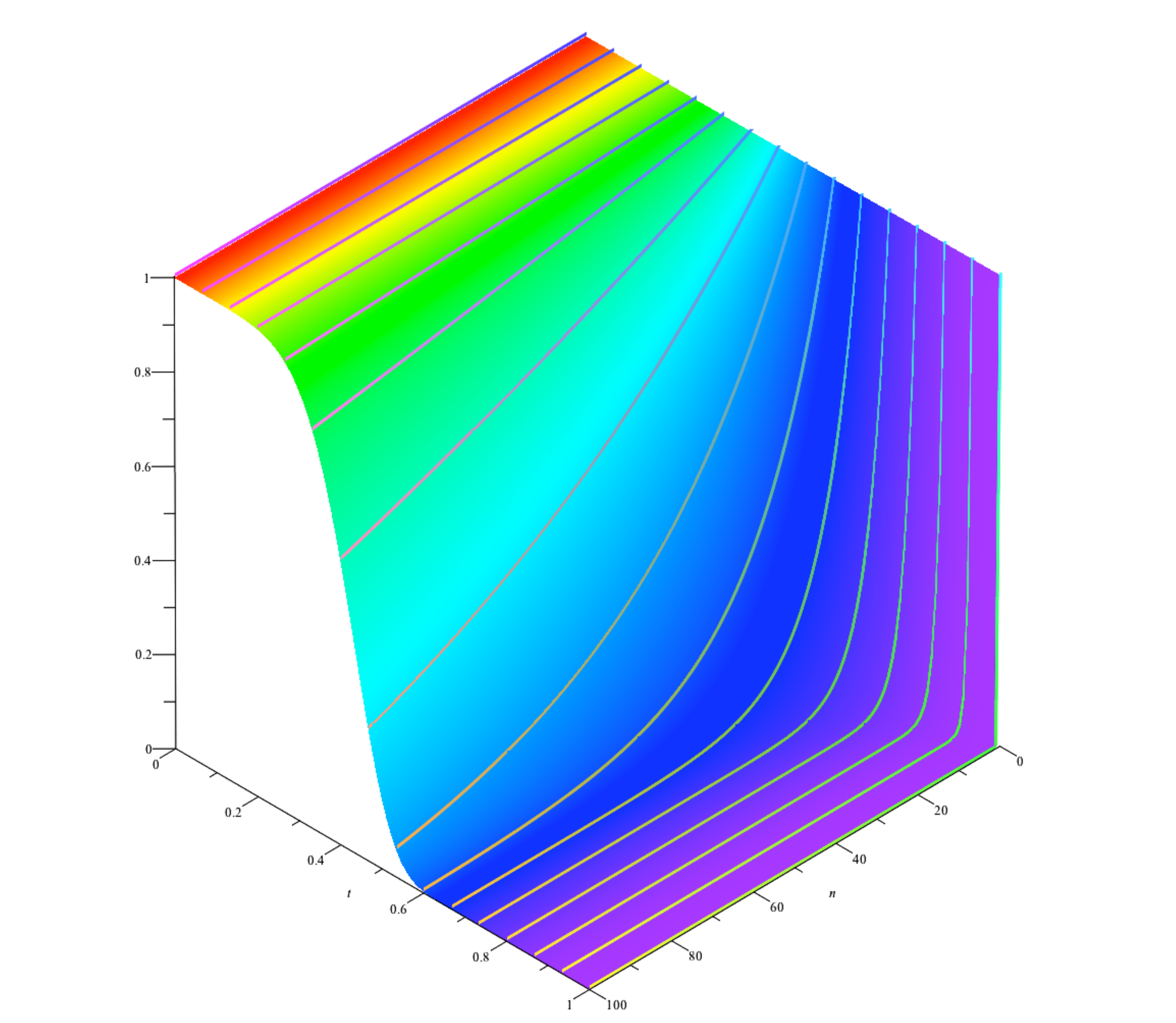}
\caption{
The distance of the first 100 terms of the sequence
$((P_VP_U)^nx)_\nnn$ to the unique solution when the angle ranges
between $0$ and $\pi/2$. 
\label{f2b}}
\end{center}
\end{figure}

In Figure~\ref{f3}, we depict the ``shadow sequence'' $(\|P_UT^nx\|)_\nnn$.
Observe again the ``rippling'' phenomenon. While the situation of two lines
appears at first to be quite special, it turns out that the same
``rippling'' also arises in a quite different setting; see \cite[Figures 4 and 6]{BK13}.

\begin{figure}[H]
\begin{center}
\includegraphics[height=3.0in]{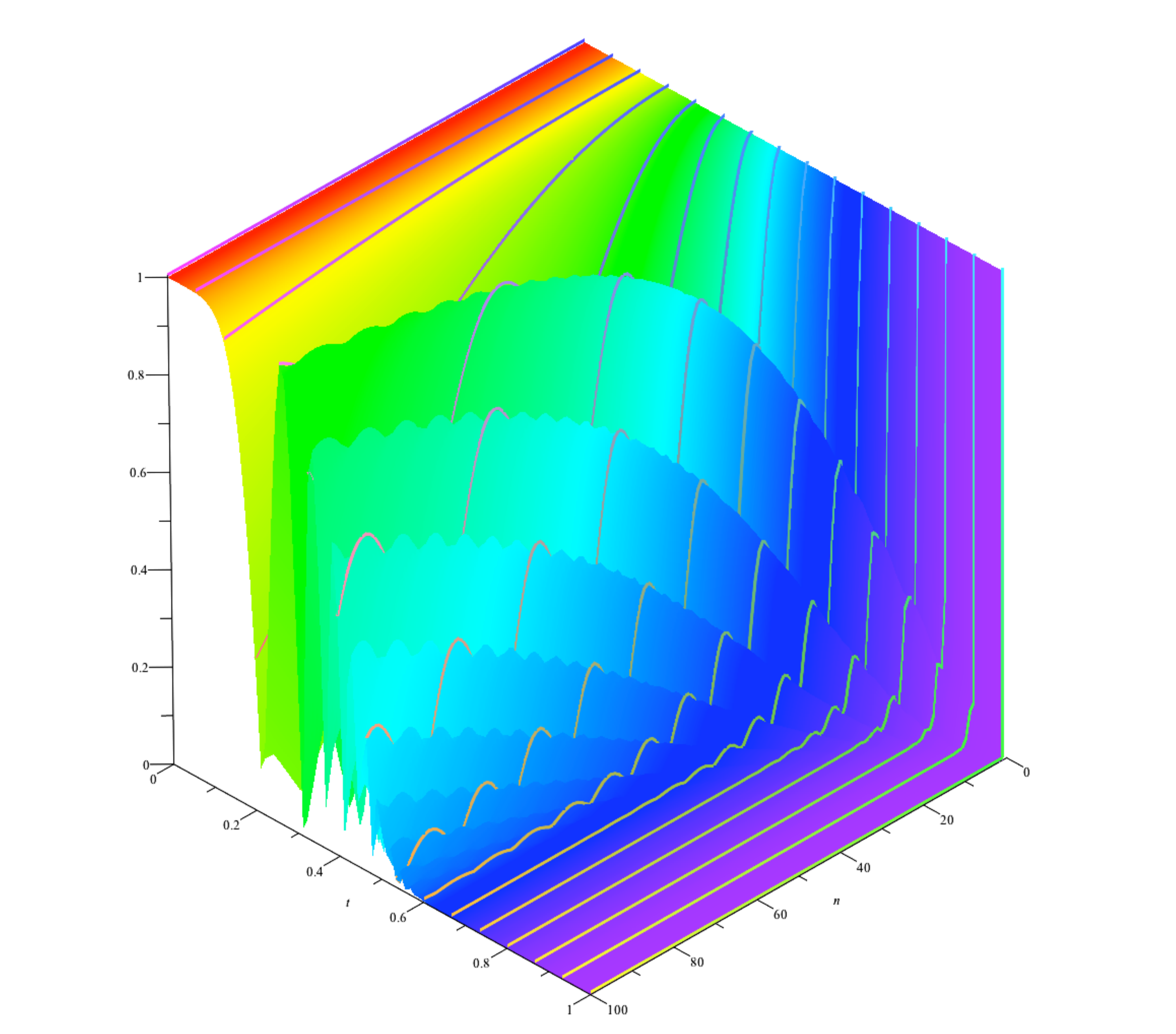}
\caption{
The distance of the first 100 terms of the ``shadow'' sequence
$(P_UT^nx)_\nnn$ to the unique solution when the angle ranges
between $0$ and $\pi/2$. 
\label{f3}
}
\end{center}
\end{figure}

The figures in this section were prepared in \texttt{Maple}\texttrademark 
\ (see \cite{Maple}). 

\section{Numerical experiments}

\label{s:numerical}

In this section, we compare 
the Douglas--Rachford method (DRM) to the method 
of alternating projections (MAP)
for finding $P_{U\cap V}x_0$. 
Our numerical set up is as follows.
We assume that $X=\RR^{50}$, and we randomly 
generated $100$ pairs of subspaces $U$ and $V$ of $X$ such that $U\cap
V\neq\{0\}$. 
We then chose 10 random starting points, each with Euclidean norm $10$.
This resulted in a total of 1,000 instances for each algorithm. 
Note that the sequences to monitor are
\begin{equation}
\big(P_UT^nx_0\big)_\nnn
\quad\text{and}\quad
\big((P_VP_U)^nx_0\big)_\nnn
\end{equation}
for DRM and for MAP, respectively. 
Our stopping criterion tolerance was set to 
\begin{equation}
\ve:=10^{-3}.
\end{equation}
We investigated two different stopping criteria, which we detail 
and discuss in the following two sections.

\subsection{Stopping criterion based on the true error}
We terminated the algorithm when the current iterate 
of the monitored sequence $(z_n)_\nnn$ satisfies
\begin{equation}
d_{U\cap V}(z_n)<\ve
\end{equation}
for the first time. 
Note that in applications, we typically would not have access to
this information but here we use it to see the true performance
of the two methods.

\begin{figure}[htb]
\hspace*{0.1in}
\includegraphics[width=5.8in]{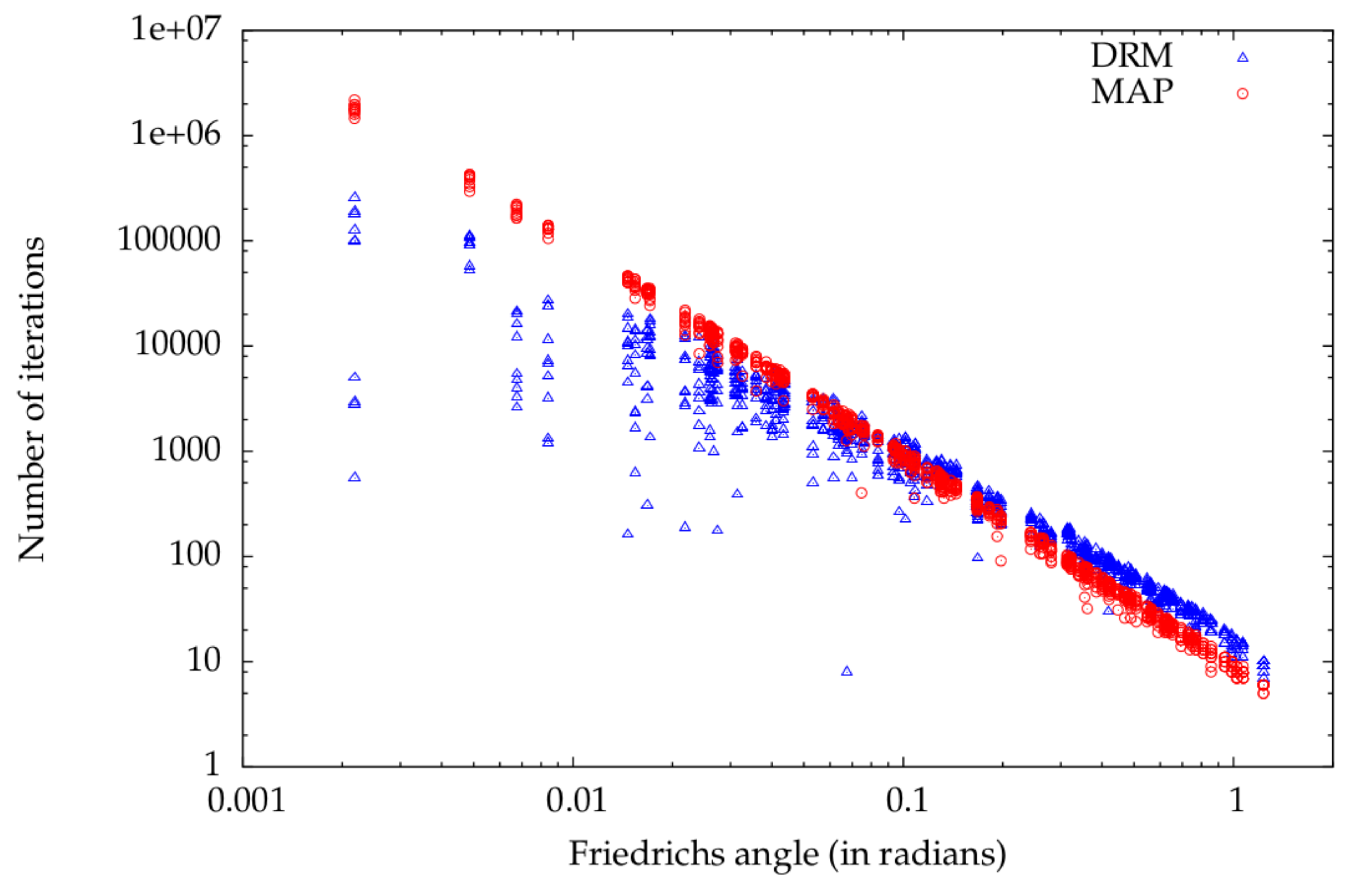}
\caption{True error criterion}
\label{f:TAll}
\end{figure}
\begin{figure}[!h]
\hspace*{0.1in}
\includegraphics[width=5.8in]{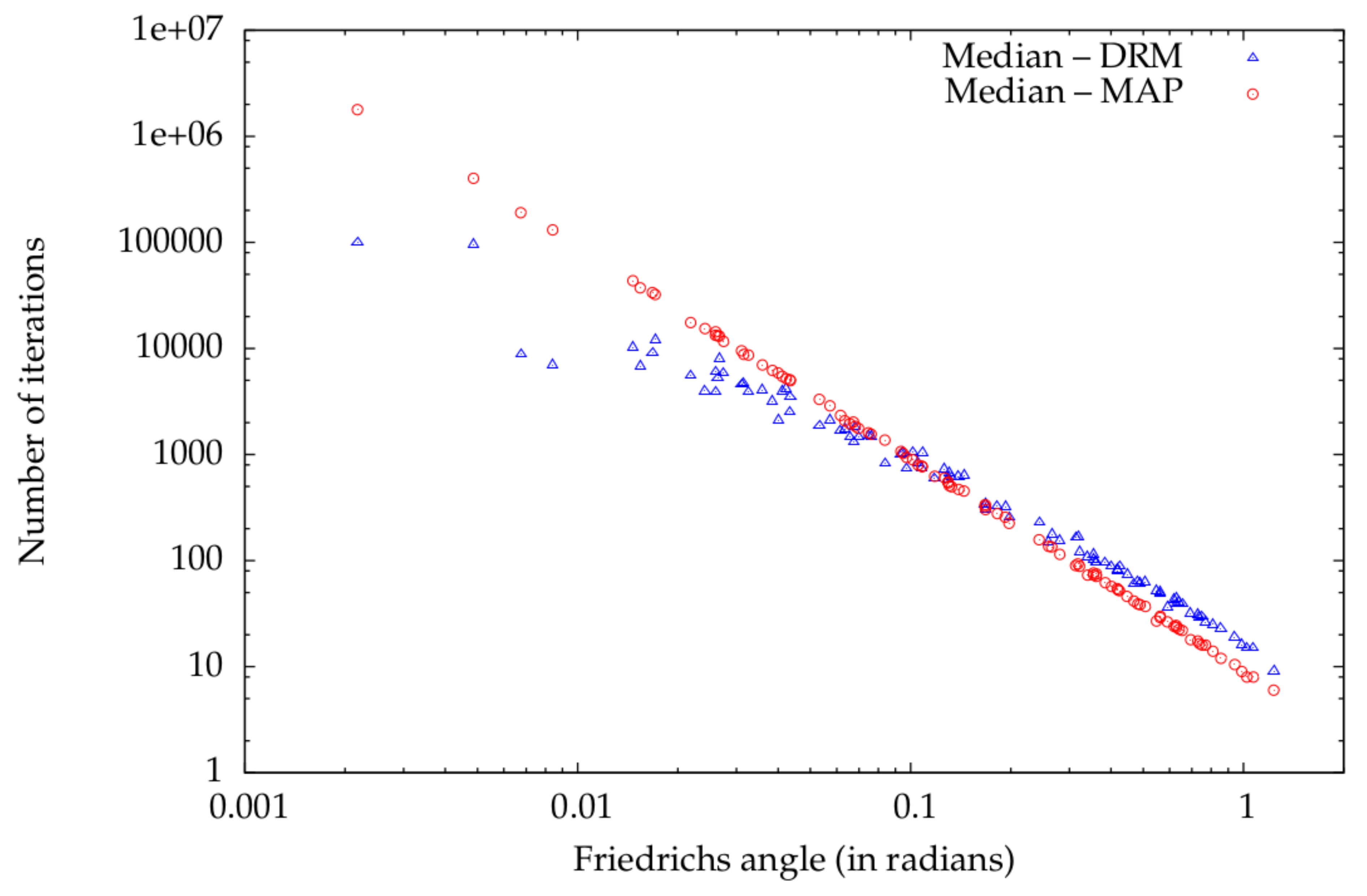}
\caption{True error criterion}
\label{f:TMed}
\end{figure}

In Figure~\ref{f:TAll} and Figure~\ref{f:TMed}, the horizontal axis
represents the Friedrichs angle between the subspaces and the vertical axis
represents the number of iterations. Results for all 1,000 runs are
presented in Figure~\ref{f:TAll}, while we show the {\em median}
in Figure~\ref{f:TMed}.

From the figures, we see that DRM is generally faster 
than MAP 
when the Friedrichs angle $\theta<0.1$. 
In the opposite case, MAP is faster.
This can be interpreted as follows. Since DRM converges with
linear rate $c_F = \cos(\theta)$ while MAP does with rate $c_F^2$, we expect that
MAP performs better when $c_F$ is small, i.e., $\theta$ is large. 
But when the Friedrichs angle is small, the ``rippling'' behaviour of DRM
appears to manifest itself (see also Figure~\ref{f1}).
Note that MAP is not much faster than DRM, which suggests DRM as the better
overall choice.

\subsection{Stopping criterion based on individual distances}

In practice, it is not always possible to obtain the true error. 
Thus, we utilized a reasonable alternative stopping criterion, namely when
the monitored sequence $(z_n)_\nnn$ satisfies
\begin{equation}
\max\big\{d_U(z_n),d_V(z_n)\big\} < \ve
\end{equation}
for the first time. 

\begin{figure}[!h]
\hspace*{0.1in}
\includegraphics[width=5.8in]{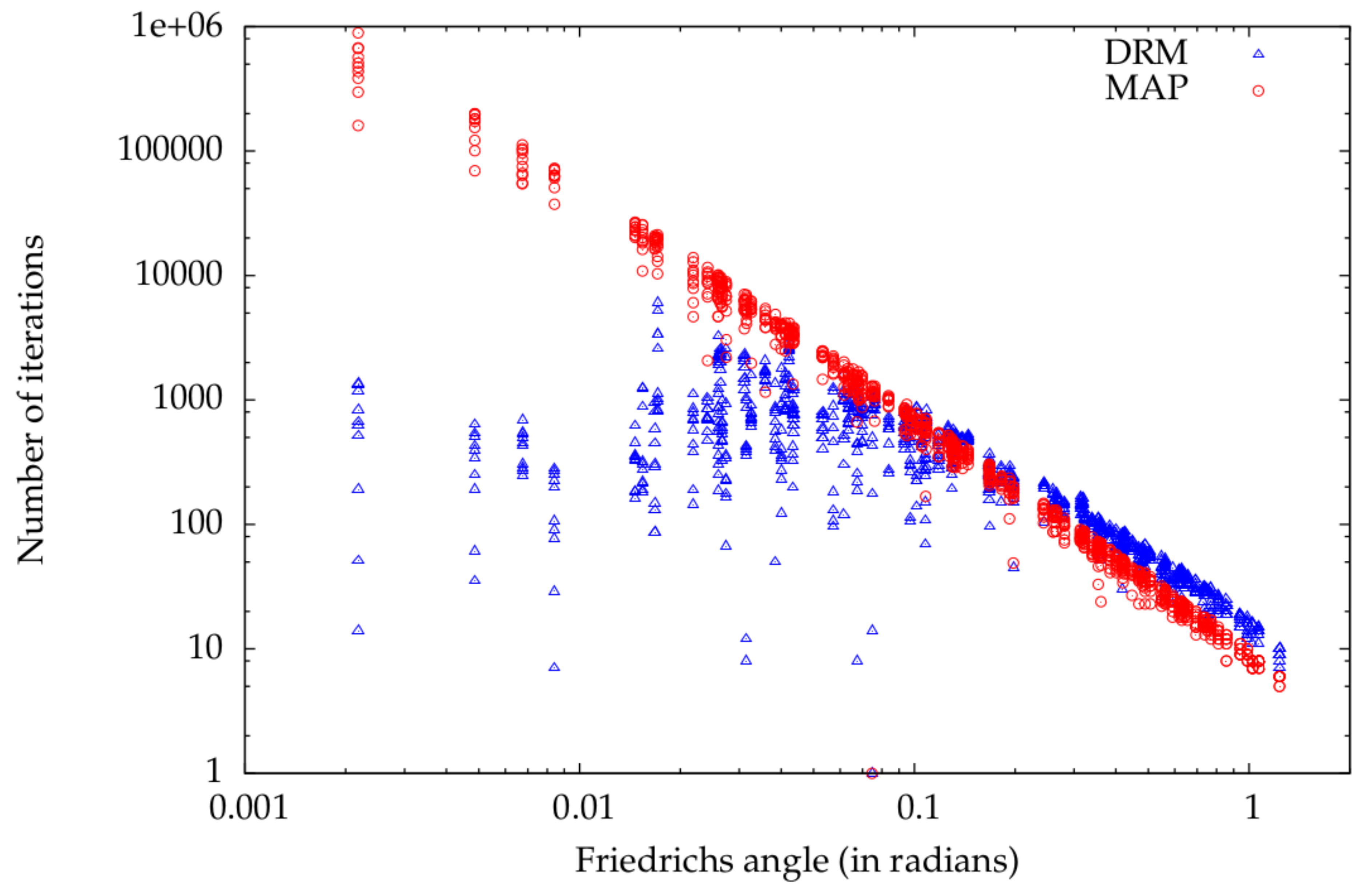}
\caption{Max distance criterion}
\label{f:MAll}
\end{figure}

\begin{figure}[!h]
\hspace*{0.1in}
\includegraphics[width=5.8in]{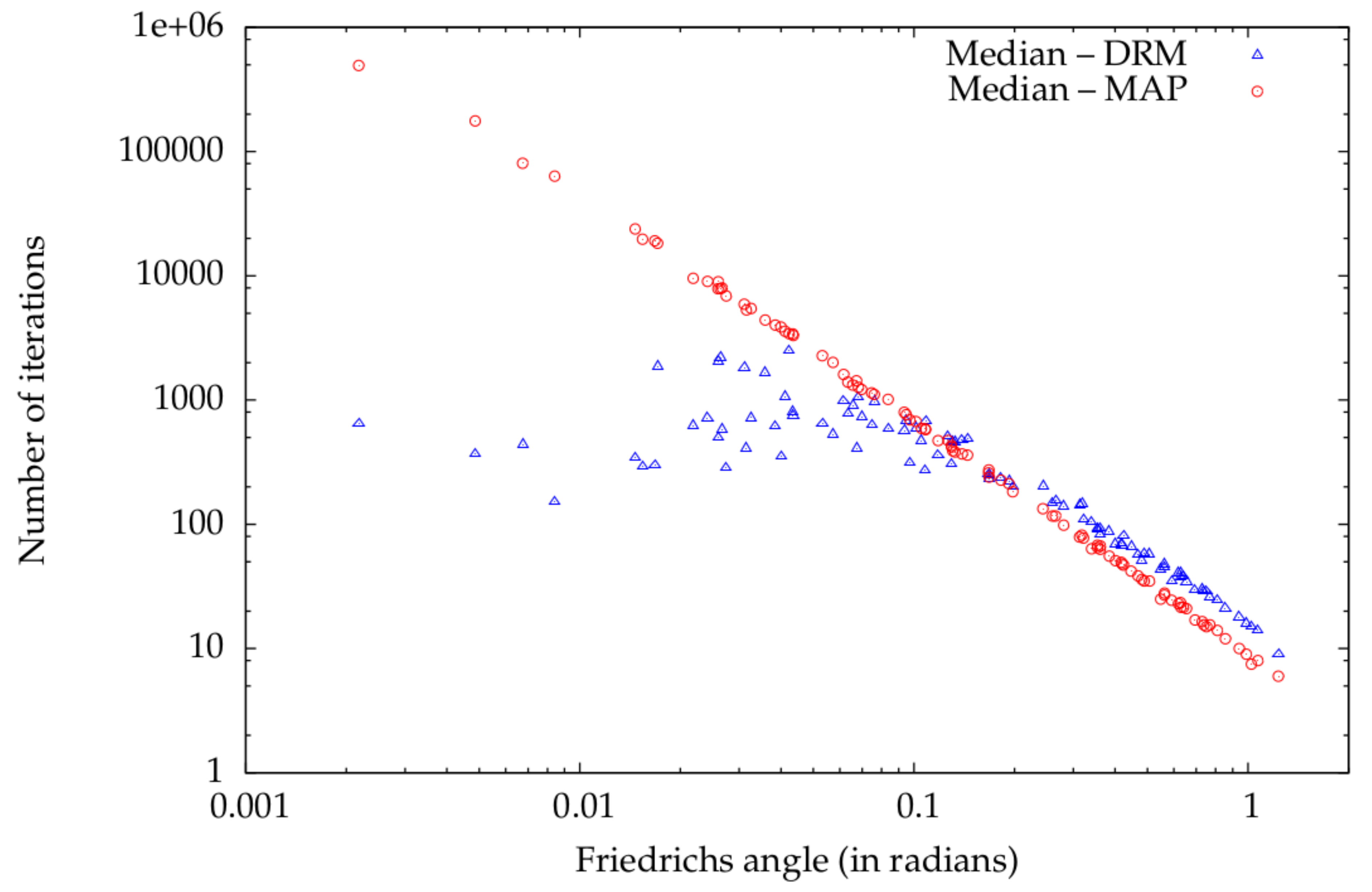}
\caption{Max distance criterion}
\label{f:MMed}
\end{figure}

Figures~\ref{f:MAll} and \ref{f:MMed} show the results 
when we use the max distance criterion with the same data.
The behaviour is similar to the experiments with the true error criterion. 

The figures in this section were computed with the help of 
\texttt{Julia} (see \cite{Julia}) and \texttt{Gnuplot} (see
\cite{Gnuplot}). 

\section{Conclusion}

\label{s:conclusion} 

We completely analyzed the Douglas--Rachford splitting method for the
important case of two subspaces. We determined the limit and the sharp rate
of convergence. Lack of linear convergence was illustrated by an example in
$\ell_2$. Finally, we compared this method to the method of alternating
projections and found the Douglas--Rachford method to be faster when the
Friedrichs angle between the subspaces is small.

\section*{Acknowledgments}

HHB was partially supported by a Discovery Grant and an Accelerator
Supplement of the Natural Sciences and Engineering
Research Council of Canada (NSERC) and by the Canada Research Chair Program.
JYBC was partially supported by CNPq and by projects UNIVERSAL and CAPES-MES-CUBA 226/2012. TTAN was partially supported by a postdoctoral 
fellowship of the Pacific Institute for the Mathematical Sciences
and by NSERC grants of HHB and XW. 
HMP was partially supported by NSERC grants of HHB and XW. XW was partially supported by a Discovery Grant of NSERC.


\small


\begin{thebibliography}{999}

\bibitem{ABT}
F.J.\ Arag\'on Artacho, J.M.\ Borwein, and M.K.\ Tam,
Recent results on Douglas-Rachford methods for combinatorial 
optimization problems,
\emph{Journal of Optimization Theory and Applications}, in press. 
\verb+DOI 10.1007/s10957-013-0381-x+

\bibitem{Aron}
N.\ Aronszajn,
Theory of reproducing kernels,
\emph{Transactions of the AMS}~68 (1950), 337--404. 

\bibitem{BaiBruRei}
J.B.\ Baillon, R.E.\ Bruck and S.\ Reich,
On the asymptotic behavior of nonexpansive mappings and semigroups
in Banach spaces, 
\emph{Houston Journal of Mathematics}~4(1) (1978), 1--9.

\bibitem{BauschkeJonFest}
H.H.\ Bauschke,
New demiclosedness principles for (firmly) nonexpansive
operators, 
\emph{Computational and Analytical Mathematics}, Chapter~2,
Springer, 2013. 

\bibitem{bb96} 
H.H.\ Bauschke and J.M.\ Borwein,
On projection algorithms for solving convex feasibility problems,
\emph{SIAM Review}~38(3) (1996), 367--426. 

\bibitem{BBHM}
H.H.\ Bauschke, R.I. Bo\c{t}, W.L.\ Hare, and W.M.\ Moursi,
Attouch--Th\'era duality revisited: paramonotonicity and operator
splitting,
\emph{Journal of Approximation Theory}~164 (2012), 1065--1084. 


\bibitem{BC2011}
H.H.\ Bauschke and P.L.\ Combettes,
\emph{Convex Analysis and Monotone Operator Theory in Hilbert Spaces},
Springer, 2011.

\bibitem{JOSA}
H.H.\ Bauschke, P.L.\ Combettes, and D.R.\ Luke,
Phase retrieval, error reduction algorithm, and 
Fienup variants: a view from convex optimization, 
\emph{Journal of the Optical Society of America} 19(7) (2002), 
1334--1345. 

\bibitem{BCL04}
H.H.\ Bauschke, P.L.\ Combettes, and D.R.\ Luke,
Finding best approximation pairs relative to two closed convex sets in
Hilbert spaces, 
\emph{Journal of Approximation Theory} 127 (2004), 
178--192. 

\bibitem{BDH}
H.H.\ Bauschke, F.\ Deutsch, and H.\ Hundal, 
Characterizing arbitrarily slow convergence in the method of
alternating projections, 
\emph{International Transactions in Operational Research}~16
(2009), 413--425. 


\bibitem{BDHP}
H.H.\ Bauschke, F.\ Deutsch, H.\ Hundal, and S.-H.\ Park,
Accelerating the convergence of the method of alternating projections,
\emph{Transactions of the AMS}~355(9)
(2003), 
3433--3461.


\bibitem{BK13}
H.H.~Bauschke and V.R.~Koch,
Projection Methods: Swiss Army Knives for Solving Feasibility and
Best Approximation Problems with Halfspaces,
in \emph{Infinite Products and Their Applications}, in press. 




\bibitem{BT13jota}
J.M.~Borwein and M.K.~Tam,
A cyclic Douglas--Rachford iteration scheme,
\emph{Journal of Optimization Theory and Applications}, in press.
\verb+DOI 10.1007/s10957-013-0381-x+


\bibitem{BruRei}
R.E.\ Bruck and S.\ Reich,
Nonexpansive projections and resolvents of accretive operators in Banach
spaces,
\emph{Houston Journal of Mathematics}~3(4) (1977), 459--470.


\bibitem{CenZen}
Y.\ Censor and S.A.\ Zenios,
\emph{Parallel Optimization},
Oxford University Press, 1997.

\bibitem{Comb93}
P.L.\ Combettes,
The foundations of set theoretic estimation,
\emph{Proceedings of the IEEE}~81(2) (1993), 182--208. 


\bibitem{Maratea}
{F.\ Deutsch},
The angle between subspaces of a Hilbert space,
in \emph{Approximation Theory, Wavelets and Applications},
S.P.\ Singh (editor), Kluwer, 1995, pp.~107--130.


\bibitem{Deutsch}
{F.\ Deutsch},
\emph{Best Approximation in Inner Product Spaces},
Springer, 2001.



\bibitem{DougRach}
J.\ Douglas and H.H.\ Rachford,
On the numerical solution of heat conduction problems
in two and three space variables,
\emph{Transactions of the AMS}~82 (1956), 421--439. 


\bibitem{EckBer}
J.\ Eckstein and D.P.\ Bertsekas,
On the Douglas-Rachford splitting method
and the proximal point algorithm for maximal monotone
operators,
\emph{Mathematical Programming (Series A)}~55 (1992), 293--318.

\bibitem{Elser}
V.\ Elser, I.\ Rankenburg, and P.\ Thibault, 
Searching with iterated maps, 
\emph{Proceedings of the National Academy of Sciences}~104(2)
(2007), 418--423.


\bibitem{Gnuplot}
GNU Plot, \texttt{http://sourceforge.net/projects/gnuplot}



\bibitem{HL}
R.\ Hesse and D.R.\ Luke, 
Nonconvex notions of regularity and convergence
of fundamental algorithms for feasibility problems, 
\emph{SIAM Journal on Optimization}~23 (2013), 2397--2419. 

\bibitem{HLN}
R.\ Hesse, D.R.\ Luke, and P.\ Neumann,
Projection methods for sparse affine feasibility: 
results and counterexamples,
\texttt{http://arxiv.org/abs/1307.2009}, July 2013. 


\bibitem{Julia}
The Julia language, \texttt{http://julialang.org/}

\bibitem{KW} 
S.\ Kayalar and H.\ Weinert,  
Error bounds for the method of alternating projections,
\emph{Mathematics of Control, Signals, and Systems}~1 (1996), 43--59.

\bibitem{LM}
P.-L.\ Lions and B.\ Mercier, Splitting algorithms for the sum
of two nonlinear operators, 
\emph{SIAM Journal on Numerical Analysis}~16 (1979), 964--979. 

\bibitem{Maple}
Maple, \texttt{http://www.maplesoft.com/products/maple/}


\bibitem{a-loch}
B.F.\ Svaiter, 
On weak convergence of the Douglas-Rachford method,
\emph{SIAM Journal on Control and Optimization}~49 (2011),
280--287. 


\end{thebibliography}
\end{document}